\documentclass[11pt]{amsart}
\usepackage[T1]{fontenc}
\usepackage{lmodern}
\usepackage[a4paper,margin=28mm]{geometry}
\usepackage{amsmath,amssymb,amsthm,booktabs,tikz}
\usepackage[hidelinks]{hyperref}
\newtheorem{theorem}{Theorem}[section]
\newtheorem{lemma}[theorem]{Lemma}
\newtheorem{proposition}[theorem]{Proposition}

\theoremstyle{remark}

\newcommand{\ceil}[1]{\left\lceil #1\right\rceil}
\newcommand{\floor}[1]{\left\lfloor #1\right\rfloor}
\newcommand{\grid}{\Gamma}
\title[Unit diagonals in odd square arrays]{Nonintersecting Unit Diagonals in Odd Square Arrays}
\author{Luyu Song}
\address{School of Mathematical Sciences, Peking University, Beijing, China}
\date{October 15, 2026}
\hypersetup{pdftitle={Nonintersecting Unit Diagonals in Odd Square Arrays},pdfauthor={Luyu Song}}

\begin{document}
\begin{abstract}
Let $D(n)$ be the maximum number of unit-square diagonals in an $n\times n$
square array such that no two selected diagonals share an endpoint and no
unit square contains both diagonals. We prove that, for every odd positive
integer $n$,
$D(n)=n(n+1)/2+\floor{(n+1)/6}$.
The matching lower bound was established by Boyland, Roth, Pint\'er,
Lauk\'o, Schoenfield, and Wasielewski. We obtain the upper bound by
transforming any admissible selection $M$ on $V_{2k-1}$ into an isometric
path cover of the $k\times k$ vertex grid with at most $2k^2-|M|$
paths. An auxiliary matching decomposes the original vertices into
alternating paths. On each such path, all segments between blocks
are either of type $da$ or of type $bc$. Projecting these paths to
the block grid gives coordinate-monotone sequences, which become
geodesics when diagonal steps are expanded. The Fisher--Fitzpatrick theorem then yields the sharp
upper bound. We also describe a Lean verification of the upper-bound
reduction, with the grid-cover theorem supplied as an explicit hypothesis.
\end{abstract}
\maketitle

\section{Introduction}

Consider the lattice vertex set
\[
 V_n=\{0,1,\ldots,n\}^2
\]
and the $n\times n$ array of unit squares that it determines. A
\emph{unit diagonal} is a diagonal of one of these squares. A selection
is \emph{admissible} if its members have pairwise disjoint endpoints
and no square contains both of its diagonals. These conditions are
equivalent to the selected segments being pairwise disjoint. Indeed,
the relative interiors of distinct unit squares are disjoint, and a
diagonal meets the boundary of its own square only at its endpoints.
Write $D(n)$ for the maximum cardinality of an admissible selection.

Boyland, Roth, Pint\'er, Lauk\'o, Schoenfield, and Wasielewski
\cite{boyland} proved
\[
 D(n)=\frac{n(n+1)}2\qquad\text{when $n$ is even}.
\]
For odd $n$, Proposition~5 and the border extension following its
proof give the lower bound
\begin{equation}\label{eq:known-lower}
 D(n)\ge \frac{n(n+1)}2+\floor{\frac{n+1}{6}}.
\end{equation}
They also gave an upper bound and discussed the conjectured optimality
of their constructions. The odd-order formula is still labeled a
conjecture in OEIS entry A264041 as accessed on October 15, 2026
\cite{oeis}. Our contribution is a general upper bound matching
\eqref{eq:known-lower}.

\begin{theorem}\label{thm:main}
For every integer $k\ge1$,
\begin{equation}\label{eq:main-k}
 D(2k-1)=2k^2-\ceil{\frac{2k}{3}}.
\end{equation}
Equivalently, for every odd positive integer $n$,
\begin{equation}\label{eq:main-n}
 D(n)=\frac{n(n+1)}2+\floor{\frac{n+1}{6}}.
\end{equation}
\end{theorem}

The upper bound applies to every admissible selection, without imposing
a prescribed pattern. Its main ingredient is a transformation to an
isometric path cover of a smaller grid. The number of paths in the
cover is controlled exactly by the number of unmatched original
vertices. Section~\ref{sec:upper} proves this transformation, including
the type invariant needed for geodesicity, and combines the resulting
upper bound with the published lower bound to prove
Theorem~\ref{thm:main}. Finally,
Section~\ref{sec:lean} states the scope of the accompanying Lean
verification.

\section{Grid paths and the Fisher--Fitzpatrick theorem}\label{sec:grids}

Coordinates are written as $(r,c)$, with the row coordinate increasing
downward and the column coordinate increasing to the right. Thus a
unit diagonal joins $(r,c)$ to $(r+1,c+1)$ or joins $(r,c+1)$ to
$(r+1,c)$. An admissible selection is a matching in the graph of all
unit diagonals, subject to the additional prohibition on the two
diagonals of a single square.

Let
\[
 \grid_k=P_k\mathbin{\square}P_k
\]
be the Cartesian grid with vertex set $\{0,\ldots,k-1\}^2$ and edges
joining vertices at Manhattan distance one. Here $P_k$ has $k$
vertices. A path is \emph{isometric}, or \emph{geodesic}, if it is a
shortest path between its endpoints. A path consisting of one vertex
is permitted. An \emph{isometric path cover} is a family of such paths
whose vertex sets cover the graph; different paths may share vertices.

The distance between $(i,j)$ and $(i',j')$ in $\grid_k$ is
$|i-i'|+|j-j'|$. In particular, an orthogonal grid path whose two
coordinates are each monotone is isometric: its horizontal and
vertical step counts are the absolute changes in the corresponding
coordinates, so its length equals this distance.

\begin{theorem}[Fisher--Fitzpatrick]\label{thm:ff}
The minimum number of paths in an isometric path cover of $\grid_k$
is $\ceil{2k/3}$.
\end{theorem}

This is the square-grid specialization of the grid-cover results in
\cite[Section~4]{fisher}; see also \cite{manuel} for the distinction
between covers and partitions. We use only the assertion that every
isometric path cover contains at least $\ceil{2k/3}$ paths. The grid
in this statement has $k^2$ vertices, rather than $k^2$ unit squares.

\section{An upper bound by local routing}\label{sec:upper}

Fix $k\ge1$ and an arbitrary admissible selection $M$ on $V_{2k-1}$.
Put $m=|M|$. Partition the $4k^2$ original vertices into the disjoint
four-vertex blocks
\begin{equation}\label{eq:blocks}
 B_{ij}=\{(2i+u,2j+v):u,v\in\{0,1\}\},\qquad 0\le i,j<k.
\end{equation}
Within a block, write
\[
 a_{ij}=(2i,2j),\quad b_{ij}=(2i,2j+1),\quad
 c_{ij}=(2i+1,2j),\quad d_{ij}=(2i+1,2j+1).
\]
We suppress the subscripts when discussing a single block and display
its labels as $\begin{smallmatrix}a&b\\c&d\end{smallmatrix}$.
An edge of $M$ is \emph{internal} if its endpoints belong to the same
block, and \emph{external} otherwise. The only possible internal edges
are $ad$ and $bc$. They cannot both be selected, since they are the
two diagonals of one unit square.

\subsection{Local routes and their ports}

Define an auxiliary perfect matching $N$ within the blocks. In a block
with no internal edge of $M$, choose $ad$ and $bc$ as the two
$N$-edges. In a block containing either internal edge, choose $ab$
and $cd$. In each case $N$ pairs all four vertices and has no edge
in common with $M$.
Here $N$ is a matching in the graph-theoretic sense: its edges may
cross in the drawing, and such crossings do not create graph vertices.

The graph consisting of the local $N$-edges and the internal
$M$-edges is a union of paths. We call these paths \emph{local
routes}. Their endpoints are called \emph{ports}. A route is of
\emph{type A} if its ports are $a,d$, oriented from $a$ to $d$, and
of \emph{type B} if its ports are $c,b$, oriented from $c$ to $b$.
All possibilities are listed in Table~\ref{tab:local}.

\begin{table}[ht]
\centering
\begin{tabular}{llll}
\toprule
Internal edge of $M$ & Edges of $N$ & Oriented local routes & Types\\
\midrule
none & $ad,bc$ & $a\to d$, $c\to b$ & A and B\\
$ad$ & $ab,cd$ & $c\to d\to a\to b$ & B\\
$bc$ & $ab,cd$ & $a\to b\to c\to d$ & A\\
\bottomrule
\end{tabular}
\caption{Local routes. In the last two rows the intermediate vertices
are already matched by the internal edge of $M$.}\label{tab:local}
\end{table}

The routes partition the original vertices. At least one route occurs
in every block. Furthermore, an external $M$-edge can meet a local
route only at a port: each intermediate vertex of a three-edge route
is already incident with its internal $M$-edge.

\begin{figure}[ht]
\centering
\begin{tikzpicture}[scale=.85,every node/.style={font=\small}]
\foreach \x/\cas in {0/0,4/1,8/2}{
 \begin{scope}[xshift=\x cm]
  \coordinate(a) at (0,1.25);\coordinate(b) at (1.25,1.25);
  \coordinate(c) at (0,0);\coordinate(d) at (1.25,0);
  \ifnum\cas=0
   \draw[dashed,thick] (a)--(d) (b)--(c);
   \node at (.625,-.6){No internal edge};
  \else
   \draw[dashed,thick] (a)--(b) (c)--(d);
   \ifnum\cas=1
    \draw[very thick] (a)--(d);\node at (.625,-.6){Internal edge $ad$};
   \else
    \draw[very thick] (b)--(c);\node at (.625,-.6){Internal edge $bc$};
   \fi
  \fi
  \foreach \p in {a,b,c,d}{\fill (\p) circle (1.4pt);}
  \node[left] at (a){$a$};\node[right] at (b){$b$};
  \node[left] at (c){$c$};\node[right] at (d){$d$};
 \end{scope}
}
\end{tikzpicture}
\caption{Solid edges belong to $M$ and dashed edges belong to $N$.
A crossing without a dot is not a graph vertex. The auxiliary
matching is used only for routing.}\label{fig:local}
\end{figure}
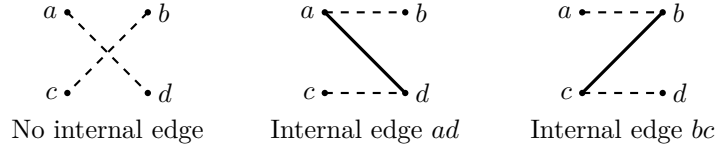

It is essential to distinguish the type of a route from the labels
of its intermediate vertices. For example, $c-d-a-b$ is a type B
route even though it contains $a$ and $d$. Those two vertices cannot
carry external edges. Type preservation concerns the ports through
which routes are connected.

\subsection{Alternating paths in the original graph}

Let
\[
 H=(V_{2k-1},M\cup N).
\]
Since $N$ is a perfect matching, $M$ is a matching, and the two
matchings have no common edge, every vertex of $H$ has degree one
or two. A degree-two vertex is incident with one edge of each
matching. Each connected component is therefore an alternating
path or an alternating cycle. In particular, $H$ has no isolated
vertices. The local routes in Table~\ref{tab:local} are the parts
of these components inside individual blocks, before external
$M$-edges are included.

\begin{lemma}[Uniform type and directed steps]\label{lem:type}
Every external edge of $M$ joins the exit port of one local route
to the entry port of another local route of the same type.
Consequently, all local routes in a connected component of $H$
have a common type.

Orient the local routes as in Table~\ref{tab:local}, and orient
each external edge from its exit port to its entry port. These
orientations make every component of $H$ a directed path or a
directed cycle. In a type A component, every external edge is
directed from $d_{ij}$ to $a_{i'j'}$, with block increment
\begin{equation}\label{eq:A-steps}
 (i'-i,j'-j)\in\{(1,0),(0,1),(1,1)\}.
\end{equation}
In a type B component, every external edge is directed from
$b_{ij}$ to $c_{i'j'}$, with block increment
\begin{equation}\label{eq:B-steps}
 (i'-i,j'-j)\in\{(-1,0),(0,1),(-1,1)\}.
\end{equation}
Thus all cross-block segments on a given component are of type
$da$, or all are of type $bc$.
\end{lemma}

\begin{proof}
A unit diagonal changes both coordinate parities. Thus it pairs
an $a$-vertex with a $d$-vertex, or a $b$-vertex with a $c$-vertex.
External edges meet local routes only at ports. An external $ad$
edge therefore joins two type A routes, and an external $bc$ edge
joins two type B routes. In either case, it joins an exit to an
entry. Every component of $H$ consists of local routes joined by
external edges, so its local routes all have the same type.

At the entry port of a local route, its first edge points out of
that port, and an external edge, if present, points into it. At
the exit port the directions are reversed. The matching property
allows at most one external edge at each port. At an intermediate
vertex of a local route there is already one incoming and one
outgoing edge, and no external edge can occur there. The orientations
therefore fit together to give directed paths and directed cycles.

The four unit-diagonal neighbors of
$d_{ij}=(2i+1,2j+1)$ are
\[
 a_{ij},\quad a_{i+1,j},\quad a_{i,j+1},\quad a_{i+1,j+1},
\]
whenever the corresponding indices are in range. The first neighbor
is internal; the other three give \eqref{eq:A-steps} in the
orientation $d\to a$. Similarly, the four possible neighbors of
$b_{ij}=(2i,2j+1)$ are
\[
 c_{ij},\quad c_{i-1,j},\quad c_{i,j+1},\quad c_{i-1,j+1}.
\]
Removing the internal neighbor gives \eqref{eq:B-steps} in the
orientation $b\to c$.
\end{proof}

\begin{lemma}[Path decomposition and count]\label{lem:count}
Every connected component of $H$ is an alternating path with at
least one edge. The number of these paths is
\begin{equation}\label{eq:components}
 p=2k^2-m.
\end{equation}
\end{lemma}

\begin{proof}
By Lemma~\ref{lem:type}, a component has a common type. Assign
the integer $i+j$ to each vertex in a type A local route in block
$B_{ij}$, and assign $j-i$ to each vertex in a type B local route
in that block. Along an oriented local route this value is constant.
Along an external edge it increases strictly, by
\eqref{eq:A-steps} and \eqref{eq:B-steps}.

Any cycle in $H$ would be directed by Lemma~\ref{lem:type} and
would contain an external edge: inside any single block, the local
routes are disjoint paths. The assigned integer would then be
nondecreasing around the cycle and would increase strictly at least
once, which is impossible. Hence every component is a path. It
has at least one edge because $N$ meets every original vertex.

A vertex of $H$ has degree one exactly when it is not matched by
$M$. There are $4k^2-2m$ such vertices. Each path component has
exactly two degree-one vertices, so
\[
 p=\frac{4k^2-2m}{2}=2k^2-m.
\]
\end{proof}

\subsection{Projected paths in the block grid}

Let $Q=\grid_k$ be the ordinary $k\times k$ vertex grid: its
vertices are the block indices $(i,j)$, and two vertices are adjacent
exactly when their Manhattan distance is one. Define the block
projection
\[
 \pi:V_{2k-1}\longrightarrow V(Q),\qquad
 \pi(r,c)=\bigl(\floor{r/2},\floor{c/2}\bigr).
\]

For a directed path $P$ of $H$, project its vertex sequence by
$\pi$ and delete consecutive repetitions. Each remaining step
corresponds to an external edge of $M$. A step may be diagonal
in the block coordinates; expand it into an east step followed
by its vertical step, as specified below. We call the resulting
walk the \emph{projected walk} of $P$ in $Q$. The next lemma shows
that it is a geodesic path.

The paths used to cover $Q$ are these projections of the paths
of $H$. Different projected paths may meet at a grid vertex;
the prescribed continuation of each is inherited from its path
in $H$. Their number is counted in $H$, not by the connected
components of their union in $Q$.

\begin{lemma}[Geodesic cover transformation]\label{lem:cover}
An admissible selection $M$ on $V_{2k-1}$ determines an isometric
path cover of $Q=\grid_k$ with at most $2k^2-|M|$ paths.
\end{lemma}

\begin{proof}
Let $P$ be a directed path component of $H$, and project it as
described above. Lemma~\ref{lem:type} applies to the entire path:
all its cross-block segments are of type $da$, or all are of type
$bc$. In the first case, both block coordinates are nondecreasing;
in the second, the row coordinate is nonincreasing and the column
coordinate is nondecreasing. After consecutive repetitions have
been deleted, every step changes at least one coordinate strictly.
Thus a block index cannot recur in the resulting sequence.

Expand its diagonal steps using
\begin{align*}
 (i,j)\to(i+1,j+1)
 &\quad\rightsquigarrow\quad
 (i,j)\to(i,j+1)\to(i+1,j+1),\\
 (i,j)\to(i-1,j+1)
 &\quad\rightsquigarrow\quad
 (i,j)\to(i,j+1)\to(i-1,j+1).
\end{align*}
Each inserted vertex lies in $Q$, because its coordinates are
coordinates of the two original endpoints. The expanded walk
uses only edges of $Q$ and preserves both coordinate directions.

This monotonicity proves simplicity and shortestness. If a vertex
occurred twice, both coordinates would have the same value at those
occurrences. Monotonicity would force them to remain constant
between those occurrences, contradicting the unit steps. The
horizontal step count equals the total increase of the column
coordinate, and the vertical step count equals the absolute change
of the row coordinate. Hence the length equals the Manhattan
distance between the endpoints, so the projected walk is a geodesic
path. If $P$ lies in a single block, its projection is the
single-vertex path at that block index.

Every original vertex belongs to a path component of $H$. Thus
each block is visited by at least one of these paths. Projection
retains its block index, and expansion only adds vertices.
Consequently, the projected paths cover all $k^2$ vertices of $Q$.
They form an indexed family with one member for each of the
$p=2k^2-m$ paths of $H$, by Lemma~\ref{lem:count}. Intersections
between members are allowed. If coincident members are discarded,
the number of paths can only decrease.
\end{proof}

\begin{proposition}\label{prop:upper}
For every $k\ge1$,
\[
 D(2k-1)\le2k^2-\ceil{\frac{2k}{3}}.
\]
\end{proposition}

\begin{proof}
For any admissible $M$, Lemma~\ref{lem:cover} and
Theorem~\ref{thm:ff} imply
\[
 \ceil{\frac{2k}{3}}\le2k^2-|M|.
\]
Rearranging and maximizing over $M$ proves the assertion.
\end{proof}

\begin{proof}[Proof of Theorem~\ref{thm:main}]
For $k\ge1$, the lower bound of Boyland et al.\ recorded in
\eqref{eq:known-lower}, with $n=2k-1$, gives
\[
 D(2k-1)\ge 2k^2-k+\floor{\frac{k}{3}}
 =2k^2-\ceil{\frac{2k}{3}},
\]
where the equality follows from
$\ceil{2k/3}=k-\floor{k/3}$.
Together with Proposition~\ref{prop:upper}, this proves
\eqref{eq:main-k}. Substituting $n=2k-1$ gives
\eqref{eq:main-n}.
\end{proof}

\section{Formal verification and its scope}\label{sec:lean}

A Lean~4.19.0 formalization of the upper-bound reduction is provided
as ancillary material with this article on arXiv. It is available
through the ancillary-file links on the article's abstract page.
The source project includes a pinned mathlib dependency, a toolchain
specification, and a README giving reproduction instructions and the
correspondence between formal declarations and the paper's arguments.
The formalization verifies the encoding of unit diagonals, the
uniform route type, acyclicity, the component count, and the
transformation to a cover by shortest grid paths. The implementation
uses local routes as an intermediate representation of the paths of $H$.

The Fisher--Fitzpatrick grid-cover bound is an explicit hypothesis
of the formal upper-bound theorem; that published result is not proved
in the Lean project. The published lower-bound construction of
Boyland et al.\ is also not formalized, and the formal
exact-value statement assumes the existence of an attaining selection.
Thus the formal verification covers the reduction and its conditional
consequences, rather than every ingredient of
Theorem~\ref{thm:main}. A separate Python script in the ancillary
material checks explicit configurations attaining the known lower
bound for all odd side lengths
from 1 through 99. These finite checks supplement the proofs and do not
replace them.

\section*{Disclosure of generative AI assistance}

Generative AI tools (OpenAI Codex) assisted with the development and
exposition of the proof, the drafting and revision of the manuscript,
and the Lean formalization. The author is responsible for the
mathematical claims, attribution, and submitted content. The scope
of the formal verification is stated in Section~\ref{sec:lean}.

\end{document}